\documentclass[11pt]{amsart}
\usepackage{amssymb,latexsym, amsmath, amscd, array, graphicx}

\usepackage{color}

\newcommand{\ad}{\mathrm{ad}}

\newcommand{\supp}{\mathrm{supp}}
\newcommand{\diam}{\mathrm{diam}}
\newcommand{\ve}{\varepsilon}

\newcommand{\pp}{{\mathcal P} }

\theoremstyle{plain}

\newtheorem{thm}{Theorem}[section]
\newtheorem{theorem}[thm]{Theorem}
\newtheorem{lem}[thm]{Lemma}
\newtheorem{lemma}[thm]{Lemma}

\newtheorem{prop}[thm]{Lemma}
\newtheorem{cor}[thm]{Corollary}

\newcommand\di[1]{\hbox{$#1$-dim}}
\newcommand\diko[1]{\hbox{$#1$-KOdim}}

\theoremstyle{definition}
\newtheorem{defin}[thm]{Definition}

\newtheorem{rem}[thm]{Remark}
\newtheorem{remark}[thm]{Remark}

\newtheorem{ex}[thm]{Example}

\newtheorem{question}[thm]{Question}

\newcommand{\topp}{{\bf top}}

\newcommand{\bott}{{\bf bot}}

\DeclareMathOperator{\dist}{{\rm dist}}

\DeclareMathOperator{\as}{{\rm asdim}}

\def\mod{\protect\operatorname{mod}}

\def\scr{\mathcal}

\def\Z{{\mathbb Z}}

\def\R{{\mathbb R}}
\def\N{{\mathbb N}}

\def\la{\langle}
\def\ra{\rangle}

\long\def\forget#1\forgotten{} %

\newcommand{\iv}{^{-1}}
\newcommand{\DG}{{\mathcal D}}

\newcommand\ver[1]{\marginpar{\tiny Changed in Ver \VER}}

\date{\today}
\begin{document}
\title[Dimension growth
]{On the dimension growth of groups}

\author[A.~Dranishnikov]{Alexander  Dranishnikov$^{1}$}

\author[M.~Sapir]{ Mark Sapir$^{2}$}

\thanks{$^{1}$Supported by NSF, grant DMS-0904278, $^{2}$Supported by NSF,
grant DMS-0700811. Both authors are thankful to the Max-Planck Institute fur Mathematik in Bonn for the hospitality.}

\address{Alexander N. Dranishnikov, Department of Mathematics, University
of Florida, 358 Little Hall, Gainesville, FL 32611-8105, U.S.A.}
\email{dranish@math.ufl.edu}

\address{Mark Sapir, Department of Mathematics, Vanderbilt University, Nashville, TN 37240,
 U.S.A.} \email{m.sapir@vanderbilt.edu}

\subjclass[2000]
{Primary 20F69; 
}

\keywords{Dimension growth, asymptotic dimension, Thompson group, {wreath product}}

\begin{abstract}
The (asymptotic) dimension growth functions of groups were introduced by Gromov in 1999. In this paper, we 
show connections between dimension growth
and expansion properties of graphs, Ramsey theory and the Kolmogorov-Ostrand dimension of groups and prove that all solvable subgroups of the R.Thompson group $F$ have polynomial dimension growth. 
We introduce controlled dimension growth function and prove that the exponentially controlled dimension growth is exponential for the Thompson group $F$ and some solvable of class 3 groups. The paper contains many open questions.
\end{abstract}

\maketitle

\tableofcontents

\section{Introduction}

Gromov introduced the notion of asymptotic dimension \cite{Gr1} to study finitely generated groups. It turns out that
groups with finite asymptotic dimension satisfy many famous conjectures like the Novikov Higher Signature Conjecture
\cite{Yu1}, \cite{Ba}, \cite{CG}, \cite{Dr1}, \cite{DFW}, \cite{BR}.
Many of the popular types of groups have finite asymptotic dimension. Such are hyperbolic groups \cite{Gr1}
 virtually polycyclic groups (hence nilpotent groups), and
solvable groups with finite rational Hirsch length \cite{BD},
\cite{DS}, Coxeter groups \cite{DJ}, arithmetic subgroups of algebraic groups over $\mathbb Q$ \cite{Ji}, any finitely
generated linear group over a field of positive
characteristic \cite{GTY}, relatively hyperbolic groups  with parabolic subgroup of finite asymptotic dimension
\cite{Os}, mapping class groups
\cite{BBF}, group acting ``nicely" on finite dimensional CAT(0) cubical complexes \cite{W}, etc.
Examples of asymptotically infinite dimensional groups include wreath product $\Z\wr \Z$ (and every other wreath
product $A\wr B$ where $A$ is not torsion and $B$ is infinite), Thompson group $F$, Grigorchuk's groups, Gromov's group
containing an expander. The infinite dimensionality of the wreath products and $F$ easily follows from the fact that
for every $n$, these groups contain $\Z^n$ as a subgroup. Grigorchuk group does not contain $\Z$ since it is a torsion
group. It is infinite dimensional because for every $n$ it coarsely contains $\R_+^n$ \cite{Sm}. This argument does not
apply to Gromov's groups containing expanders, since they can have finite cohomological dimension \cite{Gr2}; the
infinite dimensionality of them follows from the fact that these groups do not satisfy Yu's property A defined below
(some of them do not even coarsely embed into Hilbert spaces) while all groups of finite asymptotic dimension satisfy
property A \cite{HR}, see also Corollary \ref{cor:1} below.

The dimension theoretic approach still could be useful in the case of asymptotically infinite dimensional groups. Thus,
in \cite{Gr56} the notion of asymptotic dimension growth was introduced. It was shown in \cite{Dr2} that the groups with
polynomial asymptotic dimension growth have property A and some examples of groups with infinite asymptotic dimension and polynomial dimension growth were constructed. In particular, the Novikov conjecture is true for them.
Recently, answering our question, Ozawa extended this result to all groups with subexponential dimension growth
\cite{Ozawa}.

Property A (already mentioned) was introduced by Guoliang Yu \cite{Yu2}. It can be viewed as non-equivariant version of
amenability. Recall that a group (or a metric space) $X$ satisfies property $A$ if for every $\varepsilon$ and $R$
there exists a function $\xi$ from $X$ to $\ell_1(X)$ with $\xi(x)\ge 0$,  $||\xi(x)||=1$, $x\in \supp \xi(x)\subseteq
B(x,S)$ for some $S=S(\varepsilon,R)$ for every $x\in X$, so that for every two points $x,y\in X$ at distance at most
$R$, $||\xi(x)-\xi(y)|| \le \varepsilon$. A group $X$ is amenable if one can choose the functions $\xi(x)\in \ell_1(X)$
with the above property to be shifts of the properly rescaled indicator function of a bounded set of elements
containing the identity (a F\"olner set) by elements of the group $X$.

Although most known groups do have property A, for some groups like the R. Thompson group $F$ property A is still not
known and is considered almost as hard as amenability. Property $A$ implies coarse embedability of a group into the
Hilbert space.  In fact, for finitely presented groups, it is still an open question whether it is equivalent to coarse
embedability into a Hilbert space. In view of D. Farley's result \cite{Fa}, the R.Thompson group $F$ is coarsely (in
fact even equivariantly) embeddable into the Hilbert space. The compression number of such embeddings was computed in
\cite{AGS}, and unfortunately the answer lies exactly on the border
(=1/2) where it does not allow to derive property A \cite{GK}. Note that low compression number of a group does not
imply high dimension growth. For example, groups constructed in \cite{ADS} have finite asymptotic dimension and
compression number 0.

Thus the question about dimension growth of the R. Thompson group  is related to the famous amenability problem of
$F$.

\begin{defin} Let $\lambda$ be a positive number. We say that two points $x,y$ in a metric space $X$ are $\lambda$-{\em
connected} if there exists a sequence of points $x=z_0, z_1,...,z_n=y$ such that the distance between any two
consecutive points is at most $\lambda$ (that sequence will be called a $\lambda$-{\em path}). We call a set of points
$\lambda$-{\em cluster} if every two points in that set are $\lambda$-connected.
\end{defin}

\begin{defin}  Let $\lambda$ and $D$ be  positive numbers, $X$ be a metric space.  We say
that $\di{(\lambda,D)}(X)\le n$ if $X$ can be colored in at most $n+1$ colors so that every monochromatic
$\lambda$-cluster has diameter at most $D$.
\end{defin}

A coloring satisfying the above condition will be called a $(\lambda,D)$-{\em coloring} of $X$ in $n+1$ colors. The
maximal monochromatic $\lambda$-clusters of the coloring will be simply called its {\em $\lambda$-clusters}.

\begin{remark}\label{r:90} Note that we do not insist that every point is colored in only one color. Nevertheless if a
point is colored in several colors, we can pick any of them, making the point colored in only one color. As a result,
the sizes of clusters can get only smaller. So we can always assume, when dealing with the dimension growth that we
color the metric space $X$ in such a way that each point is colored in one color. In Section \ref{above}, we shall
consider a version of this function for which it will be essential that points are colored in several colors.
\end{remark}

As usual, we say that two non-decreasing functions $f,g:\R_+\to\R_+$ have the same {\em growth} if there are positive
constants $a,c$ such that $f(at)\ge g(t)-c$ and $g(at)\ge f(t)-c$ for all $t$. Clearly, this is an equivalence relation
on the set of all monotone functions. The equivalence class of a function $f$ is called the {\em growth} of that
function.

We denote by $d_X(\lambda)=\inf\{\di{(\lambda,D)}(X)\mid D\in\R_+\}$.

\begin{defin} The growth of the function $d_X(\lambda)$ is
called {\em the dimension growth} of $X$.

For a function $D=D(\lambda)$ the growth of the function $$d_{D,X}(\lambda)=\di{(\lambda,D(\lambda))}(X)$$ is called the
{\em $D$-controlled dimension growth} of $X$.
\end{defin}

{The upper limit $\limsup_{\lambda\to\infty}d_X(\lambda)$ is called {\em the
asymptotic dimension} of $X$ and is denoted  $asdim(X)$. Note that $asdim(X)$ takes values in $\N_+\cup\{\infty\}$.}

The upper limit $\limsup_{\lambda\to\infty}\di{(\lambda,D(\lambda))}(X)$ is called the $D(\lambda)$-{\em controlled
asymptotic dimension} of $X$. In particular, when $D$ is linear it is called the linearly controlled dimension or the
{\em Assouad-Nagata dimension} \cite{A}. Note that the study of asymptotic dimension with control functions was proposed by Gromov in~\cite{Gr1}.
{Some research in this direction mostly motivated by the Assouad-Nagata dimension
was done in~\cite{Brod}.}

Note that the definition of dimension function in \cite{Dr2} is similar but different: the {\em asymptotic dimension
growth} $\ad_X(\lambda)$ from \cite{Dr2} is the minimal dimension of the nerve of a uniformly bounded cover of $X$ with
the Lebesgue number $\ge\lambda$. By taking a $\lambda/2$-enlargement of a colored cover with $\lambda$-{\em disjoint}
colors (i.e. the diameters of all $\lambda$-clusters are uniformly bounded) one can construct a cover of the same
multiplicity with the Lebesgue number $\ge\lambda/2$. This yields the inequality $\ad_X(\lambda/2)\le d_X(\lambda)$.
Therefore, the growth of $\ad_X$ does not exceed the growth of $d_X$.

\begin{question} Is  the opposite inequality true, i.e. is the dimension growth equal to the asymptotic dimension
growth (for groups or general metric spaces)?
\end{question}

Certainly the answer is ``yes" when $\ad_X$ is a constant. In that case both functions give alternative definitions of
the asymptotic dimension. Thus, the functions $d_X$ and $\ad_X$ generalize two definitions of asymptotic dimension to
metric spaces with infinite asymptotic dimension.

\begin{remark}\label{rem:Z} Note that the function $\di{(\lambda,D)}(X)$ substantially depends on the control function $D$. For
example, the asymptotic dimension of $\Z$ with control $\lambda+1$ is equal to 1, but
$\di{(\lambda,\sqrt{\lambda})}(\Z)\approx \sqrt{\lambda}$ and $\di{(\lambda,0)}(\Z)=\lambda$.

{Thus, it makes little sense to consider control $D \le \lambda$.} Note also that if the control function is constant the dimension growth is equivalent to the volume growth (see Lemma \ref{dim-vs-vol} below).
\end{remark}
{\begin{question}\label{asdim} Is exponential control sufficient for detection the asymptotic dimension of a finitely generated  group?
{In more formal way, is it true that for every finitely generated group $G$
$$
\limsup_{\lambda\to\infty}\di{(\lambda,D)}(G)= \as(G)
$$
for some exponential function $D$?
}

\end{question}

Another dimension function was defined in ~\cite{CFY}: the asymptotic dimension growth $f(\lambda)$ of a metric space
$X$ is the infimum over all $n$  for which there is a uniform bounded cover $\scr U$ such that, for every $x\in X$, the
ball $B_r(x)$ intersects at most $n+1$ members of $\scr U$. It is easy to see that $f(\lambda/2)\le \ad_X(\lambda)\le
f(2\lambda)$ and hence $f$ and $\ad_X$ have the same growth.

The dimension growth, the controlled dimension growth, and asymptotic dimension growth ($d_X$, $d_{D,X}$ for all $D$,
and $\ad_X$) are quasi-isometry invariants (if we identify equivalent control functions $D$) and therefore they are
invariants of finitely generated groups.

In this paper we study the dimension growth of wreath products and the R.Thompson group $F$. The dimension growth of
$F$ turns out to be exponential for an exponential control. Note that this does not imply the exponential asymptotic
dimension growth of $F$ although the answer to the following question {(related to Question~\ref{asdim})} is unknown.

\begin{question} Are there finitely generated groups (metric spaces) $X$ with exponentially controlled exponential
dimension growth and subexponential (uncontrolled) dimension growth $d_X$? In particular, are there finitely generated
groups with finite asymptotic dimension and super-linearly (or even exponentially) controlled exponential dimension
growth.
\end{question}



{\bf Acknowledgement.} The authors are grateful to Robin Chapman, Victor Guba,  Justin Moore, Dmitri Panov, and Will Sawin for
helpful remarks and to Alexander Olshanskii and Denis Osin for spotting some mistakes in the previous version.

\section{Preliminaries}

\subsection{Volume growth and dimension growth}

We recall that the {\em chromatic number} of a graph is the minimal number of colors (if exists) such that the vertices
of the graph can be colored in a way that adjacent vertices have different colors.
\begin{prop}\label{coloring}
Let $K$ be a possibly infinite graph of valency $\le c$. Then its chromatic number $\le c+1$.
\end{prop}
\begin{proof}
Using the Zorn lemma, take a maximal $(c+1)$-colorable induced subgraph $K'$ of $K$. Suppose that $K'\ne K$. Any vertex
$v$ of $K$ that is not in $K'$ has at most $c$ colored neighbors and hence it can be colored and added to $K'$. That
contradicts the maximality of $K'$.
\end{proof}

\begin{remark}
We note that for $D\ge D'$, $d_{D,X}\le d_{D',X}$, also $d_X\le d_{D,X}$ for all $D$.
\end{remark}

A version of the next proposition for the function $\ad$ is proved in \cite{Dr2}.

\begin{prop}\label{dim-vs-vol}
The dimension growth  of a finitely generated group $G$ with any control function does not exceed its volume growth.
\end{prop}
\begin{proof} Let $f$ be the volume growth function of $G$ (relative to some finite generating set). We consider a
graph $R_\lambda(G)$ whose vertices are all elements of $G$ where every two vertices at distance $\le\lambda$ are
joined by an edge (this is of course the 1-skeleton of the Rips complex of $G$). Then the valency of every vertex of
$R_\lambda(G)$ is $\le f(\lambda)$. By Lemma~\ref{coloring} the graph has chromatic number $\le f(\lambda)+1$.
Thus, a coloring of the graph in $f(\lambda)+1$ colors is a coloring of the Cayley graph of $G$ with monochromatic
$\lambda$-clusters of diameter 0. Note that $f(\lambda)+1$ is equivalent to $f(\lambda)$.
\end{proof}

The following corollary was mentioned in \cite{Gr56} as obvious.

\begin{cor} The dimension growth of any finitely generated group with any control function is at most
exponential.
\end{cor}

\subsection{Dimension growth and quasi-isometries}

{We recall that a map of metric spaces $\phi:X\to Y$ is called a {\em coarse embedding} if there are strictly monotone tending to
infinity functions $\rho_1,\rho_2:\R_+\to\R_+$ and a number $r>0$ such that $$
\rho_1(d_X(x,x'))\le d_Y(\phi(x),\phi(x'))\le\rho_2(d_X(x,x'))
$$ for all $x,x'\in X$ with $d(x,x')\ge r$. A typical example of a coarse embedding is an inclusion of a finitely generated subgroup in a finitely generated group
both supplied with the word metrics. In that case $r=0$, $\rho_2$ is linear, and $\rho_1$ could with any greater than linear
growth of $\rho_1^{-1}$.

\begin{prop}\label{p:1-Lip} Let $\phi:G'\to G$ be a 1-Lipschitz isomorphism of  finitely generated groups. Then $\di{\lambda}\left(G'\right)\le\di{\lambda}\left(G\right)$.
\end{prop}

\proof It is enough to note that the $\phi$-preimage of every uniformly bounded $\lambda$-disjoint family in $G$ is a uniformly bounded and $\lambda$-disjoint family in $G'$.
\endproof

\begin{prop}\label{p:6}
Let $\phi:X\to Y$ be a coarse embedding with functions $\rho_1,\rho_2$:
 Then
$$
\di{(\lambda,D)}(Y)\ge\di{(\rho^{-1}_2(\lambda),\rho_1^{-1}(D))}(X).
$$
\end{prop}
\begin{proof} We assume that $\lambda, D>r$. Suppose that $\di{(\lambda,D)}(Y)\le m$.
For any $\left(\lambda, D\right)$-coloring of $Y$ in $m$ colors $c:Y\to\{1,\dots,m\}$
the composition $c\circ \phi$ defines a coloring of $X$ in $m$ colors. Note that for any two clusters $C$ and $C'$ in $Y$ of the same color
$$
\lambda<d_Y(C,C')=\inf_{y\in C, y'\in C'}d_Y(y,y')\le \rho_2(d_X(\phi^{-1}(C),\phi^{-1}(C')))
$$
and
$$
\rho_1(diam_X(\phi^{-1}(C)))\le diam_Y(C)\le D.
$$
Therefore, $$\rho_2^{-1}(\lambda)<d_X(\phi^{-1}(C),\phi^{-1}(C'))\ \ \text{and}\ \ diam_X(\phi^{-1}(C))\le \rho_1^{-1}(D).$$
Thus, $c\circ \phi$ is a  $\left(\rho_2^{-1}(\lambda), \rho_1^{-1}(D)\right)$-coloring of $X$ bin $m$ colors. So,
$$\di{\left(\rho_2^{-1}(\lambda), \rho_1^{-1}(D)\right)}(X)\le m.$$ This completes the proof.
\end{proof}}

Applying Lemma \ref{p:6} to quasi-isometries gives

\begin{prop}\label{subgroup}
Let $\phi:X\to Y$ be a quasi-isometric embedding:
$$\frac{1}{c_1}d_X\left(x,x'\right)-r_1\le d_Y\left(\phi\left(x\right),\phi\left(x'\right)\right)\le cd_X\left(x,x'\right)+r $$ for
all $x,x'\in X$ where $c,c_1\ge 1, r,r_1\ge 0$. Then for $\lambda>r$,
$$
\di{\left(\lambda,D\right)}(Y)\ge\di{\left(\frac{\lambda-r}{c},c_1\left(D+r_1\right)\right)}(X).
$$
Thus if for some function $D$, the dimension growth of $Y$ with control $D(\lambda)$ is $d(\lambda)$, then the dimension growth of $X$ with control $c_1\left(D(c\lambda+r)+r_1\right)$ is at most $d(c\lambda+r)$.
\end{prop}

For a metric $d$ on a discrete space $X$ and $r>0$ we denote by $d+r$ the new metric $\bar d$ defined as $\bar
d(x,y)=d(x,y)+r$ provided $x\ne y$ and $\bar d(x,x)=0$. We call it {\em the metric $d$ shifted by a constant $r$}.

Lemma \ref{subgroup} immediately implies

\begin{prop}\label{shift}
For every $r,\lambda,D$ such that $0\le r< \lambda, r<D$, $$\di{(\lambda,D)}(X,d+r)=\di{(\lambda-r,D-r)}(X,d).$$
\end{prop}

\subsection{Dimension growth and expansion in graphs}

Let $G$ be a graph such that there exists a number $\varepsilon>0$ such that for every $r$ there exists a finite
subgraph $G_r\subseteq G$ with the following property
\medskip

$\left(P_r(\varepsilon)\right)$ For every subset $A$ of {vertices} of $G_r$ of diameter (in $G$) $\le r$, $|\partial_{G_r}(A)|\ge
\varepsilon |A|$ where {$\partial_{G_r}=\{v\in G_r\mid \dist(v,A)=1\}$} denotes the boundary in $G_r$.
\medskip

\begin{theorem} \label{th:exp} Suppose that {there exists $\varepsilon>0$ such that a graph} $G=(V_G,E_G)$ (where $V_G$ is the set of vertices and $E_G$ is the set of edges) satisfies $\left(P_r(\varepsilon)\right)$ for all $r$.  Then
the dimension growth of $G$ is exponential.
\end{theorem}

\proof Consider any even number $\lambda>0$. Let $V_G=\cup_{i=1}^{k+1} U_i$ be a coloring of the vertices of $G$ in $k+1$ colors such
that all $\lambda$-clusters $U_i^j$ have diameters at most $d$. Take $r>d+\lambda$ and consider the graph $G_r=(V_r,E_r)$. Let
$W_i^j=U_i^j\cap G_r$. We have $\cup W_i^j$  equal to the set $V_r$ of all vertices of $G_r$. Note that by the choice of $r$ and by property
$\left(P_r(\varepsilon)\right)$, the $\lambda/2$-neighborhood $N_{\lambda/2}(W_i^j)$ of $W_i^j$ in $G_r$ has at least
$(1+\varepsilon)^{\lambda/2}$ elements. Since different $\lambda$-clusters of the same color are $\lambda$-disjoint, we
have that the sum of cardinalities $|N_{\lambda/2}(W_i^j)|$ is at most $|V_r|(k+1)$. On the other hand, that sum is at least
$(1+\varepsilon)^{\lambda/2}$ times the sum of cardinalities $|W_i^j|$, i.e. at least $|V_r|(1+\varepsilon)^{\lambda/2}$.
Hence $k+1\ge (1+\varepsilon)^{\lambda/2}$ which implies the statement of the theorem.
\endproof

The following two corollaries of Theorem \ref{th:exp} follow from Ozawa's result \cite{Ozawa} that metric spaces of
subexponential dimension growth satisfy property A, and the result of Willett \cite{Willett} that certain graphs do not
satisfy property A.  Theorem \ref{th:exp} gives direct proofs of these corollaries. See also the recent paper \cite{W5}
where it is proved that in the case of graphs of bounded degree the condition of Theorem \ref{th:exp} implies that $G$
does not satisfy property A.

\begin{cor}\label{cor:1} Let $G$ be a graph containing an expander. Then the dimension growth of $G$ is exponential.
\end{cor}

\proof Indeed, by definition, every sufficiently large (finite) graph in an expander sequence satisfies property
$\left(P_r(\varepsilon)\right)$ for some $\varepsilon$ (which is related to the spectral gap of the expander) and all
$r$. \endproof

\begin{cor}\label{cor:2} Let $G$ be a graph such that for every $r\ge 0$ there exists a subgraph $G_r$ of $G$ with
degree of each vertex at least $3$ and with girth $> r$. Then the dimension growth of $G$ is exponential. In
particular, the dimension growth of the Gromov ``random monster" \cite{Gr2} constructed using any sequence of finite
regular graphs with increasing girth (not necessarily an expander sequence of finite graphs) is exponential. 
\end{cor}

\proof Indeed, if a graph $G_r$ has girth $\ge r$ and minimal degree of a vertex $k\ge 3$, then every subgraph $A$ of
$G_r$ of diameter $\le r$ is a forest and there are at least $|A/2|$ vertices in $A$ that are connected to at least
$(k-2)$ vertices in $G_r$ outside $A$, and each of these vertices is connected to at most $k$ vertices in $A$. Thus one
can take $\varepsilon=\frac{k-2}{2k}$.\endproof

\begin{remark} By \cite{Sapirasph} there exists a 5-dimensional aspherical Riemannian manifold whose fundamental group contains expander and hence has exponential dimension growth. That answers a question from \cite[Page 158]{Gr56}.
\end{remark}

In view of Theorem \ref{th:exp} and \cite{W5}, the following question becomes very interesting.

\begin{question}\label{q:4} Is it true that a finitely generated group (or a graph with bounded degrees of vertices)
satisfies property A if and only if it has subexponential dimension growth?
\end{question}

\section{Dimension growth of direct sums of $\Z$}

\subsection{Coloring unit cubes}
The following two examples are well known.

\begin{ex}\label{ex:Z} For every $\lambda\ge 1$, color a number $z\in\Z$ in black if $\lfloor
\frac{z}{2\lambda}\rfloor$ is even, and in white otherwise. This is a $(\lambda,2\lambda)$-coloring of $\Z$. Hence
$$\di{(\lambda,2\lambda)}(\Z)\le 1.$$
\end{ex}

\begin{ex}
Using the checker coloring of vertices of $\Z^n$ (the color of the point $(x_1,...,x_n)$ of $\Z^n$ is the sum $\sum
x_i$ modulo $2$) one gets $$\di{(1,0)} (\Z^n,\ell_1)=1.$$
\end{ex}



It will be clear later that expansion of ``small" subsets in unit cubes plays important role in computing dimension growth of groups. See also Gromov's papers \cite{Gr7} and \cite{Gr8} where a connection between this question and Shannon inequalities is considered. 

\begin{prop}\label{p:56} Let $G$ be the binary cube $\{0,1\}^n$ with the $\ell_1$-metric. Then for every $r> 0$, such
that $\varepsilon=\frac{n}{r+1}-2>0$, $G$ satisfies property $\left(P_r(\epsilon)\right)$.
\end{prop}
\begin{proof} Let $A\subseteq G$.
Induction on $d=\diam(A)$.

If $d=0$, then $A$ is a point and the inequality $|\partial A|=n>\frac{n}{r+1}-2=\left(\frac{n}{D+1}-2\right)|A|$ is
true.

Assume that $\left(P_r(\epsilon)\right)$ holds for $d-1$ and let $\diam(A)=d\le r$. We may assume that $\vec 0\in A$
where $\vec 0=(0,0,...,0)$. Let $F=A\cap S(\vec 0,d)$, where $S(x,t)$ denotes the sphere in $G=\{0,1\}^n$ of radius $t$
centered in $x$. Since the diameter of $A$ is $d$, we can assume that $F$ is not empty. By the induction assumption, $$
|\partial(A\setminus F)|>\varepsilon|A\setminus F|.$$ Let $B=\partial A\cap S(\vec 0,d+1)$. 
since the diameter of Since $\partial(A\setminus F)\subset\partial A\cup F$, we obtain $\partial(A\setminus F)\cup
B\subset\partial A\cup F$. Note that $B\cap\partial(A\setminus F)=\emptyset$. Also $\partial A\cap F=\emptyset$. Thus,
the above unions are disjoint $$\partial(A\setminus F)\sqcup B\subset\partial A\sqcup F$$ where $\sqcup$ denotes
disjoint union. Hence $$|\partial(A\setminus F)|+|B|\le|\partial A|+ |F|.$$

Note that every point $y\in B$ is obtained from some $x\in F$ by replacing one 0 coordinate by 1. Since $\|x\|=d$,
there are exactly $n-d$ coordinates of $x$ that are equal to 0. Thus each $x\in F$ corresponds to $n-d$ points in $B$.
Each $y\in G$ can be obtained from a point from $F$ at most in $d+1$ ways. Thus, $$|G|\ge\frac{(n-d)|F|}{d+1}.$$
Therefore $$ |\partial A|\ge|\partial(A\setminus F)|+|B|-|F|\ge\left(\frac{n}{r+1}-2\right)|A\setminus F|+
\frac{(n-d)|F|}{d+1}-|F|=$$
$$
\left(\frac{n}{r+1}-2\right)|A\setminus F|+\left(\frac{n}{d+1}-\frac{d}{d+1}-1\right)|F|\ge
\left(\frac{n}{r+1}-2\right)|A\setminus F|+\left(\frac{n}{d+1}-2\right)|F|$$
since $d\le r$, $$
\ge\left(\frac{n}{r+1}-2\right)|A\setminus
F|+\left(\frac{n}{r+1}-2\right)|F|=\left(\frac{n}{r+1}-2\right)|A|=\varepsilon |A|.
$$
\end{proof}
For every subset $A$ of vertices of a graph $X$ and every $r\ge 1$ we use the notation $$
\partial_rA=\{x\in X\mid 0 < d(x,A)\le r\}$$
for the $r$-boundary of a set $A$. Then $\partial A=\partial_1A$.
\begin{cor}\label{2}
For any subset $A\subset\{0,1\}^n$ of diameter $\le r\le n/4$ and $n>16$ $$ |\partial_2 A|>\frac{n^2}{4(r+2)^2}|A|.$$
\end{cor}
\begin{proof}
Note that $\partial_2A=\partial A\sqcup\partial\left(A\cup\partial A\right)$. Therefore, $$ |\partial_2A|=|\partial
A|+|\partial\left(A\cup\partial A\right)|\ge \left(\frac{n}{r+1}-2\right)|A|+
\left(\frac{n}{r+2}-2\right)(|A|+|\partial A|)$$
$$
\ge \left(\left(\frac{n}{r+2}-2\right)\left(\frac{n}{r+1}-1\right)+\left(\frac{n}{r+2}-2\right)\right)|A|\ge $$
$$\left(\left(\frac{n}{r+2}-2\right)\left(\frac{n}{r+2}-1\right)+\left(\frac{n}{r+2}-2\right)\right)|A|
=\left(\frac{n}{r+2}-2\right)\frac{n}{r+2}|A|\ge \frac{n^2}{4\left(r+2\right)^2}|A|.$$ Here we used the inequality
$\frac{n}{\left(r+2\right)}\ge \frac{8}{3}$ for $r< \frac n4$ and $n\ge 16$.
\end{proof}
We recall that a $\lambda$-cluster, $\lambda\in\N$, of $A\subset \{0,1\}^n$ is a maximal subset $C\subset A$ such
that for every pair of points $x,x'\in C$ there is a sequence $x_i\in C$ with $x_0=x$, $x_k=x'$, and
$d(x_i,x_{i+1})\le\lambda$. Thus, every two $\lambda$-clusters are {at} distance $\ge\lambda+1$. The proof of the
following lemma is similar to the proof of Theorem \ref{th:exp}.

\begin{lemma}\label{lm:cube}
The binary $n$-cube $\{0,1\}^n$, $n>64$, cannot be colored by $n$ colors such that each 4-cluster of every color has
diameter less than $\le\sqrt n/4$.
\end{lemma}

\begin{proof}
Assume that there is  such a coloring: $$\{0,1\}^n=A_1\cup\dots\cup A_n$$ with $A_i=\cup_j A^j_i$, $\diam(A^j_i)\le
r<\sqrt n/4$, each $A_i^j$ is a $4$-cluster of color $i$, and $d(A_i^j,A_i^k)\ge 5$ for all $i$, $j$, and $k\ne j$.
By Corollary~\ref{2}, $$ |\partial_2
A^j_i|>\frac{n^2}{4(r+2)^2}|A^j_i|$$ for all $i$ and $j$ where $r<\sqrt n/4$. Since $d(A_i^j,A_i^k)\ge 5$, we have
$\partial_2A_i=\sqcup_j\partial_2 A^j_i$ (disjoint union). Therefore, $$ |\partial_2 A_i|>\frac{n^2}{4(r+2)^2}|A_i|.$$
Since $\sum_{i=1}^n|A_i|=2^n$ (the number of vertices in the binary cube) we obtain, $$\sum_{i=1}^n|\partial_2
A_i|>\frac{n^2}{4(r+2)^2}2^n.$$ Since $2^n>|A_i|$, we have $n2^n>\frac{n^2}{4(r+2)^2}2^n$ or equivalently, $r+2>\sqrt
n/2$. Then $\sqrt n/4>\sqrt n/2-2$ which is equivalent to $8>\sqrt n$. This contradicts the assumption that $n>64$.
\end{proof}

\begin{cor}\label{mainL}
$\di{(4,\sqrt n/4)}(\{0,1\}^n,\ell_1)\ge n$ for  $n>64$.
\end{cor}

\begin{cor} \label{cor:cubes} If a metric space $X$ contains isometric copies of binary cubes $\{0,1\}^n$ for all $n$,
then
$\di{4}(X)=\infty$. In particular, $\di{4}(\bigoplus^{\infty}\Z)=\infty$.
\end{cor}
\begin{proof}
Assume that $\di{4}(X)=k<\infty$. This means that for some $r$, $\di{(4,r)}(X)=k$. By Corollary \ref{mainL},
$\di{(4,r)}(\{0,1\}^l)=l$ for all $l\ge\max\{4r^2,64\}$, a contradiction.
\end{proof}
\subsection{{Dimension growth and the Ramsey theory}}
Answering our question Dmitri Panov and Justin Moore ~\cite{Pa} gave two proofs that
$\di{2}(\bigoplus^{\infty}\Z)=\infty$. Here we include a proof by Justin Moore. It shows a connection between dimension
growth and the Ramsey theory. The proof can be easily adapted to any metric space that contains isometric copies of
arbitrary large binary cubes (as in Corollary \ref{cor:cubes}).

\begin{thm}\label{sum of Z} $\di{2}(\bigoplus^{\infty}\Z)=\infty$.
\end{thm}

\proof Every finite subset $M$ of $\N$ corresponds to a vector $v(M)$ from $\Z^\infty$ with coordinates $0$, $1$ in the
natural way ($v(M)$ is the indicator function of $M$). Choose any $k\ge 1$. Let $P_k(\N)$ denote the set of all
$k$-element subsets of $\N$. Every finite coloring of $\Z^\infty$ induces a finite coloring of $P_k(\N)$. By the
classic result of Ramsey \cite{GRS} there exists a subset $M\subseteq \N$ of size $2k$ such that all $k$-element
subsets of $M$ have the same color.
Therefore we can find subsets $T_1,T_2,\ldots, T_k$ of size $k$ from $M$ such that the symmetric distance between $T_i$
and $T_{i+1}$ is 2, $i=1,\ldots,k-1$, and $T_1, T_k$ are disjoint. Then the vectors $v(T_1),\ldots, v(T_k)$ belong to
the same $2$-cluster of the coloring and the diameter of that cluster is $\ge 2k$. Thus for every  coloring of
$\Z^\infty$ in finite number of colors and every $k$ there exists a $2$-cluster of diameter $\ge k$, hence $2$-clusters
must have arbitrary large diameters.  This immediately implies the statement of the theorem. \endproof

\subsection{Dimension growth and the game of Hex}

The following questions seem to be interesting and non-trivial (see Remark \ref{r:56} below).

\begin{question}\label{i} For every $k\ge 1$ let $f(k)=\di{2}(\Z^k)$. What is the rate of growth of $f$? Is this
function bounded? Is $f(k)=k$  for every $k\ge 1$? Is $f(k)\ge k^\alpha$ for some $\alpha>0$?
\end{question}

Note that this question is similar in spirit to the famous game of Hex \cite{Gale}. Recall that the $n$-dimensional Hex
board  of size $k$ consists of all vertices $z=(x_1,\ldots,x_n)\in\Z^n$ such that $1\le z_i\le k$, $i=1,\ldots,n$ that
is all vertices of an $n$-dimensional cube $I_k^n$ of size $k$. A pair of vertices $(z_1,\ldots, z_n),
(z_i',\ldots,z_n')$ is called {\em adjacent} if $\max_i(|z_i-z_i'|, 1\le i\le n)=1$ and all differences $z_i-z_i'$ are
of the same sign.  For every $i=1,\ldots,n$ let $H_i^-=\{(z_1,\ldots,z_n)\mid z_i=1\}$, $H_i^+=\{(z_1,\ldots,z_n)\mid
z_i=k\}$. The following theorem can be found, for example, in \cite{Gale}.

\begin{thm}\label{hex} For every $k,n$ and every coloring of $I_k^n$ in $n$ colors there exists a monochromatic path of
color $i$ (for some $i=1,\ldots,n$) connecting $H_i^-$ and $H_i^+$.
\end{thm}

In order to answer Question \ref{i} one needs to consider the following modified game Hex$_1$ with the same board but
calling two vertices adjacent if the $l_1$-distance between them is 1 (in the standard Hex game the distance is
$l_\infty$). It is easy to see that the function $f(k)$ from Question  \ref{i} would be equal to $k+1$ if we had a
statement similar to Theorem \ref{hex} for the game Hex$_1$.

It is also known \cite{Gale} that Theorem \ref{hex} is equivalent to the Brouwer fixed point theorem (in the sense that both theorems easily follow from each other). It would be interesting to find a fixed point-type statement which implies an answer to Question \ref{i}.

\section{Wreath products}

Let $B$ and $A$ be finitely generated groups with finite generating sets $S$ and $T$ and word metrics $|.|_B$ and
$|.|_A$ respectively. Then the (reduced) wreath product $B\wr A$ is the semidirect product of $B^{(A)}$ (the group of
all functions $C_0(A,B)=\{A\to B\}$ with finite support) and $A$ for the natural action of $A$ on $B^{(A)}$.  We define
the $\ell_1$-metric on $B^{(A)}$ as follows: $$ d_{\ell_1}(f,g)=\sum_{a\in A}d_B(f(a),g(a)).$$

For every $a\in A$ we denote by $B_a$ the group of functions $A\to B$ with the support $\{a\}$. We shall always
identify $B$ with the group $B_e$. Note that  $B_a=a\iv Ba$. If the operation in $B$ is written additively, we shall
write $\bigoplus_A B$ instead of $B^{(A)}$.

If $S$ is a generating set for $B$, and $T$ is a generating set for $A$, then $S\subset B$  together with
$T\subset A$ generate $B\wr A$. An explicit formula for the word metric $d_{B\wr A}$ on $B\wr A$ was found by Parry.

\begin{theorem}[\cite{Parry}]\label{th:Parry} Let $g=ba\in B\wr A$, where $a\in A, b\in B$. Let
$b=b_1^{a_1}...b_n^{a_n}$ for some $b_i\in B$ and $a_i\in A$ where $B$ is identified with the subgroup of $B^{(A)}$
consisting of all functions $A\to B$ with support $\{1\}$. Let $p$ be the shortest path in the Cayley graph of $A$ that
starts at 1, visits all vertices $a_i$ and ends at $a$. Then the word length of $|g|_{B\wr A}$ of $g$ in $B\wr A$ is
the length of $p$ plus $\sum |b_i|_B$.
\end{theorem}

The following statement immediately follows from Theorem \ref{th:Parry}.

\begin{cor}\label{c:Parry}
(1) For every $a\in A$ the metric on $B_a\cong B$ induced from $B\wr A$ is the metric on $B$ shifted by $2|a|$:
$d_{B\wr A}(x^a,y^a)=d_B(x,y)+2|a|$ for $x\ne y\in B$.

(2)  $d_{B\wr A}(f,g)\ge d_{\ell_1}(f,g)$ for all $f,g\in B^{(A)}$. Where $d_{\ell_1}$ is the $\ell_1$-metric on
$B^{(A)}$.
\end{cor}

We define $B_k$ to be the $k$th iterated wreath product of $\Z$. Formally, $B_0=\Z=\la b_0\ra$ and if $B_k=\la
b_0,...,b_k\ra$ is already constructed, then $$B_{k+1}=B_k\wr\Z$$ where the ``top" $\Z$ is generated by $b_{k+1}$.

By induction we define a canonical subgroup $D_k\cong
\bigoplus\Z\subset B_k$: $D_0=B_0=\Z$ and $D_{k+1}=\bigoplus_{i\in\Z}D_k$.
Note that $D_k$ is a sum of copies $\Z_{\vec i}$ of $\Z$ indexed by vectors $\vec i\in \Z^k$ with $\|\vec
i\|_1=|i_1|+\dots+|i_k|$. Let $d_k$ be the metric on $D_k$ induced from $B_k$.

\begin{prop}\label{d-shift}
The metric $d_k$ restricted to the summand $\Z\subset D_k$ indexed by $\vec i$ is the standard metric shifted by $\|\vec i\|$: $$|x|_{d_k}=|x|+\|\vec i\|_1$$ for $x\in\Z$, $x\ne 0$.
\end{prop}
\begin{proof}
Induction on $k$. Let $\Z$ be indexed by $\vec i\in\Z^{k+1}$. By Corollary~\ref{c:Parry}(1)
$|x|_{d_{k+1}}=|x|_{d_k}+|i_{k+1}|$ for $x\in\Z_{\vec i}\subset (B_k)_{i_{k+1}}\subset B_k\wr\Z$. By induction
assumption, $|x|_{d_k}=|x|+|i_1|+\dots+|i_k|$. Then $|x|_{d_{k+1}}=|x|+\|\vec i\|$.\end{proof}

\begin{prop}\label{bilipschitz}
For any $r>0$ the idenity map $$id:(\bigoplus_{\|\vec i\|_1\le r}\Z,\ell_1)\to(\bigoplus_{\|\vec i\|_1\le
r}\Z,d_k)\subset D_k$$ is $(1,r)$-bi-Lipschitz:$$
\|x-y\|_{\ell_1}\le d_k(x,y)\le r\|x-y\|_{\ell_1}.
$$
\end{prop}
\begin{proof}
In view of Lemma~\ref{d-shift} the $d_k$-norm of every vector from the standard basis of $\bigoplus_{\|\vec
i\|\le r}\Z$ does not exceed $r$. This implies that the above identity map is $r$-Lipschitz.

The inequality $\|x\|_{\ell_1}\le \|x\|_{d_k}$ will be proven by induction on $k$. Let $$x\in\bigoplus_{\vec
i\in\Z^{k+1},\|\vec i\|\le r}\Z\subset D_{k+1}=\bigoplus_{j\in\Z}D_k.$$ Thus, $x=(x_{s})$ with $x_s\in (D_k)_s$. By
Corollary~\ref{c:Parry}(2), $\|x\|_{d_{k+1}}\ge\sum_{s\in\Z}\|x_{s}\|_{d_k}$. By induction assumption,
$\|x_{s}\|_{d_k}\ge\|x_s\|_{\ell_1}$. Therefore, $$
\|x\|_{d_{k+1}}\ge\sum_{s\in\Z}\|x_s\|_{\ell_1}=\|x\|_{\ell_1}.
$$
\end{proof}

Let $P\subset\R^k$ be a polytope with integral vertices. The Ehrhart polynomial $L(P,t)$ of $P$ is defined as
$$L(P,t)=|tP\cap \Z^k|$$ where $tP$ is dilation of $P$ and $|\ |$ denotes the cardinality. Thus $L(P,t)$ is the number
of integer points in the polytope $P$ dilated by the factor $t$. It is known that $L(P,t)$ is a polynomial of degree
$k$ with positive coefficients \cite{BR1}.

The {\em regular cross-polytope} in $\R^k$ is the polytope spanned by the vertices $\{\pm e_i\mid i=1,\dots,k\}$ where
$\{e_i\}$ is the orthonormal basis of $\R^k$. The Ehrhart polynomial for the regular cross-polytope $P_k\subset\R^k$ is
known \cite{BR1}: $$ L(P_k,x)=\sum_{i=0}^k\frac{2^ix(x-1)\dots(x-i+1)}{i!}.$$

\begin{thm}\label{estimate} For each $k$ there are $\alpha$ and $\beta$ such that
$$\di{(\lambda,R)}(B_k)\ge \beta\lambda^k$$ for $R<\alpha\lambda^{\frac{k}{2}}$ for sufficiently large $\lambda$.
\end{thm}
\begin{proof} Pick numbers $r$ and $R$  such that $\frac{\lambda}{2r+1}\ge 5$ and $R<\sqrt{L\left(P_k,r\right)}/4$.
By Lemma~\ref{bilipschitz} the idenity map $$id:(\bigoplus_{\|\vec i\|_1\le r}\Z,\ell_1)\to(\bigoplus_{\|\vec
i\|_1\le r}\Z,d_k)\subset D_k$$ is $(1,r)$-bilipschitz.

Considering the subspaces and applying Lemma~\ref{subgroup} we obtain $$\di{\left(\lambda,R\right)} (B_k,d_k)\ge
\di{\left(\lambda,R\right)} (D_k,d_k)\ge \di{\left(\lambda,R\right)} (\bigoplus_{\|\vec i\|_1\le r} \Z, d_k)$$
$$\ge\di{\left(\frac{\lambda}{r},R\right)}\bigoplus_{\|\vec i\|_1\le r}\Z\ge |\{\vec i\in\Z^k\mid \|\vec i\|_1\le r\}|=
L\left(P_k,r\right).$$ Here the last equality is by definition. The preceeding inequality holds  true by
Corollary~\ref{mainL} provided   $r<\lambda/4$ and  $R<\alpha\lambda^{\frac{k}{2}}$ for some $\alpha$ which depends on
$k$ only. In that case, $\di{\left(\lambda,R\right)}(B_k)\ge\beta\lambda^{k}$ for some $\beta$ which depends on $k$
only.
\end{proof}

The proof of the following Lemma is similar to the proof of Lemma~\ref{bilipschitz}.

\begin{lem}\label{l3}
For any $r>0$ the idenity map $$id:(\bigoplus_{\|g\|_G\le r}\Z,\ell_1)\to(\bigoplus_{\|g\|_G\le r}\Z,d_{\Z\wr
G})\subset Z\wr G$$ is $(1,r)$-bi-Lipschitz.
\end{lem}


\begin{thm} \label{t3}
Let $G$ be a group of exponential growth.  Then for some exponential function $D$ the dimension growth of the group $\Z
\wr G$ with control $D$ is exponential.
\end{thm}
\begin{proof}
 Let $B_r$ be  the ball of radius $r$ in $G$.

In view of Lemma~\ref{l3} the identity map $$ Z=((\bigoplus_{g\in
B_r}\Z),\ell_1)\stackrel{i}\longrightarrow((\bigoplus_{g\in B_r}\Z),d)=X $$ satisfies $$ d_Z(z,z')\le  d_X(z,z')\le
rd_Z(z,z'). $$ Application of Lemma~\ref{subgroup} to the above quasi-isometries and Corollary~\ref{mainL} gives
us the chain of inequalities $$\di{(\lambda,R)}((\bigoplus_{g\in G}\Z),d)
\ge\di{(\lambda,R)}((\bigoplus_{g\in B_r}\Z),d)\ge
\di{\left(\frac{\lambda}{r},R\right)}(\bigoplus_{g\in B_r}\Z)\ge |B_r|$$ whenever
$$R\le\frac{\sqrt{|B_r|}}{4r}$$ and $\frac{\lambda}{r}> 4$ where $|B_r|$ denotes the  cardinality of the ball. Note
that $$\frac{\sqrt{|B_r|}}{4r}\ge \frac{e^{\beta r}}{4r}$$ for some $\beta>0$. We take $r=\lambda/5$ to satisfy the
second condition and take $\alpha$ such that $e^{\alpha\lambda}<\frac{e^{\beta\lambda/5}}{4\lambda/5}$. Then for
$R(\lambda)<e^{\alpha\lambda}$ we have the required condition.
\end{proof}

\begin{cor}\label{c:zzz} For some exponential function $D$ the dimension growth of the solvable of class 3 group $\Z\wr (\Z \wr \Z)$ with
control $D$ is exponential.
\end{cor}

\proof Indeed, the volume growth function of $\Z\wr \Z$ is exponential.
\endproof

\section{Lower bound for dimension growth of the R. Thompson group}

In this section, it will be  convenient to view the R. Thompson group as a diagram group over the semigroup
presentation $\la x\mid x^2=x\ra$.

Let us recall the definition of a {\em diagram group} (see \cite{GS1,GS2} for more formal definitions). A (semigroup)
{\em diagram} is a planar directed labeled graph tesselated into cells, defined up to an isotopy of the plane. Each
diagram $\Delta$ has the top path $\topp(\Delta)$, the bottom path $\bott(\Delta)$, the initial and terminal vertices
$\iota(\Delta)$ and $\tau(\Delta)$. These are common vertices of $\topp(\Delta)$ and $\bott(\Delta)$.  The whole
diagram is situated between the top and the bottom paths, and every edge of $\Delta$ belongs to a (directed) path in
$\Delta$ between $\iota(\Delta)$ and $\tau(\Delta)$. More formally, let $X$ be an alphabet. For every $x\in X$ we
define the {\em trivial diagram} $\varepsilon(x)$ which is just an edge labeled by $x$. The top and bottom paths of
$\varepsilon(x)$ are equal to $\varepsilon(x)$, $\iota(\varepsilon(x))$ and $\tau(\varepsilon(x))$ are the initial and
terminal vertices of the edge. If $u$ and $v$ are words in $X$, a {\em cell} $(u\to v)$ is a planar graph consisting of
two directed labeled paths, the top path labeled by $u$ and the bottom path labeled by $v$, connecting the same points
$\iota(u\to v)$ and $\tau(u\to v)$. There are three operations that can be applied to diagrams in order to obtain new
diagrams.

1. {\bf Addition.} Given two diagrams $\Delta_1$ and $\Delta_2$, one can identify $\tau(\Delta_1)$ with
$\iota(\Delta_2)$. The resulting planar graph is again a diagram denoted by $\Delta_1+\Delta_2$, whose top (bottom)
path is the concatenation of the top (bottom) paths of $\Delta_1$ and $\Delta_2$. If $u=x_1x_2\ldots x_n$ is a word in
$X$, then we denote $\varepsilon(x_1)+\varepsilon(x_2)+\cdots + \varepsilon(x_n)$ ( i.e. a simple path labeled by $u$)
by $\varepsilon(u)$  and call this diagram also {\em trivial}.

2. {\bf Multiplication.} If the label of the bottom path of $\Delta_2$ coincides with the label of the top path of
$\Delta_1$, then we can {\em multiply} $\Delta_1$ and $\Delta_2$, identifying $\bott(\Delta_1)$ with $\topp(\Delta_2)$.
The new diagram is denoted by $\Delta_1\circ \Delta_2$. The vertices $\iota(\Delta_1\circ \Delta_2)$ and
$\tau(\Delta_1\circ\Delta_2)$ coincide with the corresponding vertices of $\Delta_1, \Delta_2$, $\topp(\Delta_1\circ
\Delta_2)=\topp(\Delta_1),
\bott(\Delta_1\circ \Delta_2)=\bott(\Delta_2)$.

\begin{center} 
\unitlength=1mm
\special{em:linewidth 0.4pt}
\linethickness{0.4pt}
\begin{picture}(124.41,55.00)
\put(1.00,30.00){\circle*{2.00}}
\put(46.00,30.00){\circle*{2.00}}
\put(1.00,30.00){\line(1,0){45.00}}
\bezier{320}(1.00,30.00)(24.00,55.00)(46.00,30.00)
\bezier{332}(1.00,30.00)(24.00,5.00)(46.00,30.00)
\put(24.00,35.00){\makebox(0,0)[cc]{$\Delta_1$}}
\put(24.00,25.00){\makebox(0,0)[cc]{$\Delta_2$}}
\put(24.00,10.00){\makebox(0,0)[cc]{$\Delta_1\circ\Delta_2$}}
\put(66.00,30.00){\circle*{2.00}}
\put(94.00,30.00){\circle*{2.00}}
\put(123.00,30.00){\circle*{2.00}}
\bezier{164}(66.00,30.00)(80.00,45.00)(94.00,30.00)
\bezier{152}(66.00,30.00)(81.00,17.00)(94.00,30.00)
\bezier{172}(94.00,30.00)(109.00,46.00)(123.00,30.00)
\bezier{168}(94.00,30.00)(110.00,15.00)(123.00,30.00)
\put(80.00,30.00){\makebox(0,0)[cc]{$\Delta_1$}}
\put(109.00,30.00){\makebox(0,0)[cc]{$\Delta_2$}}
\put(94.00,10.00){\makebox(0,0)[cc]{$\Delta_1+\Delta_2$}}
\end{picture}
\end{center}

3. {\bf Inversion.} Given a diagram $\Delta$, we can flip it about a horizontal line obtaining a new diagram
$\Delta\iv$ whose top (bottom) path coincides with the bottom (top) path of $\Delta$.

\begin{defin} A diagram over a collection of cells $P$ is any planar
graph obtained from the trivial diagrams and cells of $P$ by the operations of addition, multiplication and inversion.
If the top path of a diagram $\Delta$ is labeled by a word $u$ and the bottom path is labeled by a word $v$, then we
call $\Delta$ a $(u,v)$-diagram over $P$.
\end{defin}

Two cells in a diagram form a {\em dipole} if the bottom part of the first cell coincides with the top part of the
second cell, and the cells are inverses of each other. In this case, we can obtain a new diagram removing the two cells
and replacing them by the top path of the first cell. This operation is called {
\em elimination of dipoles}. The new diagram is called {\em equivalent}
to the initial one. A diagram is called {\em reduced} if it does not contain dipoles. It is proved in \cite[Theorem
3.17]{GS1} that every diagram is equivalent to a unique reduced diagram.

 If the
top and the bottom paths of a diagram are labeled by the same word $u$, we call it a {\em spherical} $(u,u)$-diagram.
Now let $P=\{c_1,c_2,\ldots\}$ be a collection of cells. The diagram group $\DG(P,u)$ corresponding to the collection
of cells $P$ and a word $u$ consists of all reduced spherical $(u,u)$-diagrams obtained from the  cells of $P$ and
trivial diagrams by using the three operations mentioned above. The product $\Delta_1\Delta_2$ of two  diagrams
$\Delta_1$ and $\Delta_2$ is the reduced diagram obtained by removing all dipoles from $\Delta_1\circ\Delta_2$. The
fact that $\DG(P,u)$ is a group is proved in \cite{GS1}.

\begin{ex} If $X$
consists of one letter $x$ and $P$ consists of one cell $x\to x^2$, then the group $\DG(P,x)$ is the R. Thompson group
$F$ \cite{GS1}.
\end{ex}

Here are the diagrams representing the two standard generators $x_0, x_1$ of the R. Thompson group $F$. All edges are
labeled by $x$ and oriented from left to right, so we omit the labels and orientation of edges.

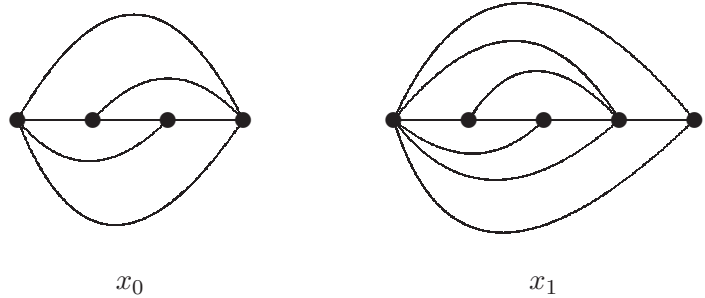
\begin{center} 
\unitlength=1mm
\special{em:linewidth 0.4pt}
\linethickness{0.4pt}
\begin{picture}(94.00,50.00)
\put(3.00,24.00){\circle*{2.00}}
\put(13.00,24.00){\circle*{2.00}}
\put(23.00,24.00){\circle*{2.00}}
\put(33.00,24.00){\circle*{2.00}}
\put(53.00,24.00){\circle*{2.00}}
\put(63.00,24.00){\circle*{2.00}}
\put(73.00,24.00){\circle*{2.00}}
\put(83.00,24.00){\circle*{2.00}}
\put(93.00,24.00){\circle*{2.00}}
\put(3.00,24.00){\line(1,0){10.00}}
\put(13.00,24.00){\line(1,0){10.00}}
\put(23.00,24.00){\line(1,0){10.00}}
\put(53.00,24.00){\line(1,0){10.00}}
\put(63.00,24.00){\line(1,0){10.00}}
\put(73.00,24.00){\line(1,0){10.00}}
\put(83.00,24.00){\line(1,0){10.00}}
\bezier{120}(13.00,24.00)(23.00,35.00)(33.00,24.00)
\bezier{120}(3.00,24.00)(12.00,13.00)(23.00,24.00)
\bezier{256}(3.00,24.00)(19.00,52.00)(33.00,24.00)
\bezier{256}(3.00,24.00)(14.00,-4.00)(33.00,24.00)
\bezier{132}(63.00,24.00)(71.00,37.00)(83.00,24.00)
\bezier{108}(53.00,24.00)(64.00,15.00)(73.00,24.00)
\bezier{208}(53.00,24.00)(70.00,45.00)(83.00,24.00)
\bezier{176}(53.00,24.00)(65.00,8.00)(83.00,24.00)
\bezier{296}(53.00,24.00)(67.00,55.00)(93.00,24.00)
\bezier{296}(53.00,24.00)(62.00,-6.00)(93.00,24.00)
\put(18.00,2.00){\makebox(0,0)[cc]{$x_0$}}
\put(73.00,2.00){\makebox(0,0)[cc]{$x_1$}}
\end{picture}
\end{center}

It is easy to represent, say, $x_0$ as a product of sums of cells and trivial diagrams: $$x_0=\left(x\to
x^2\right)\circ \left(\varepsilon\left(x\right)+\left(x\to x^2\right)\right)\circ \left(\left(x\to x^2\right)\iv
+\varepsilon\left(x\right)\right)\circ \left(\left(x\to x^2\right)\iv\right).$$

There is a natural {\em diagram metric} on every diagram group $\DG(P,u)$: $\dist(\Delta,\Delta')$ is the number of
cells in the diagram $\Delta\iv\Delta'$.

\begin{lemma}[\cite{burillo, AGS}]\label{bu} For the R. Thompson group $F$,
the diagram metric  is $(6,2)$-quasi-isometric to the word metric  corresponding to the standard generating set $\{x_0,
x_1\}$.
\end{lemma}

\begin{prop}\label{embedding}
There are constants $C_1,C_2>0$ such that for every $n$ there is a group embedding of $\xi_n:\Z^{2^n}\to F$ into the
Thompson group $F$ such that $\xi_n$ is a $(C_1n,C_2)$-quasi-isometric embedding: $$\frac{1}{C_1n}|x-x'|_1-C_2\le
d_F(\xi_n(x),\xi_n(x'))\le C_1n\|x-x'\|_1+C_2$$ where $\|.\|_1$ is the standard $l_1$-metric on $\Z^{2^n}$.
\end{prop}
\begin{proof}
We are going to use the following construction from \cite{AGS}. For any $n\ge0$, let us define $2^n$ elements of $F$
that commute pairwise. All these elements will be reduced $(x,x)$-diagrams over $\pp=\la x\mid x^2=x\ra$. For $n=0$,
let $\Delta$ be the diagram that corresponds to the generator $x_0$ (see above).  It has 4 cells.

Suppose that $n\ge1$ and we have already constructed diagrams $\Delta_i$ ($1\le i\le2^{n-1}$) that commute pairwise.
For every $i$ we consider two $(x^2,x^2)$-diagrams: $\ve(x)+\Delta_i$ and $\Delta_i+\ve(x)$. We get $2^n$ spherical
diagrams with base $x^2$ that obviously commute pairwise. It remains to conjugate them to obtain $2^n$ spherical
diagrams with base $x$ having the same property. Namely, we take $\pi\circ(\ve(x)+\Delta _i)\circ {\pi}^{-1}$ and
$\pi\circ(\Delta_i+\ve(x))\circ\pi ^{-1}$.

Let us denote the elements of $F$ obtained in this way by $g_i$ ($1\le i\le2^n$). It is easily proved, say, by
induction on $n$, that there exists a $(x^{2^n},x)$-diagram $u_n$ with $n$ cells and $(x^{2^n},x^{2^n})$-diagrams
$v_{n,i}=\ve(x^i)+\Delta+\ve(x^{2^n-i-1})$, $i=0,...,2^n-1$,  such that each $g_i$ is equal to $u_n\iv v_{n,i}u_n$.
Hence each $g_i$ has $2n+4$ cells and its word length in $F$ is bounded between $n/C$ and $Cn$ where $C$ is a constant.
Hence the subgroup $A_n$ generated by $g_1,...,g_{2^n}$ is isomorphic to $Z^{2^n}$.

Now if we consider the diagram $g_1^{k_1}...g_{2^n}^{k_{2^n}}$ for any integers $k_1,...,k_{2^n}$, the number of cells
in that diagram is between $4(|k_1|+...|k_{2^n}|)$ and $2n+4(|k_1|+...|k_{2^n}|)$. It follows from Lemma \ref{bu} that
the restriction of the word metric of $F$ on the subgroup $A_n$ is between $\frac1{C_1}|.|-C_2n$ and $C_1|.|+C_2n$
where $|.|$ is the standard $l_1$-metric on $Z^{2^n}$, $C_1, C_2$ are constants $>1$.
\end{proof}

\begin{rem} Note that the constants $C_1$ and $C_2$ in Lemma
\ref{embedding} do not exceed 25 and do not depend on $n$.
\end{rem}

\begin{thm}\label{t7}
There exists an exponential function $D$ such that the dimension growth of the Thompson group $F$ with control $D$ is
exponential.
\end{thm}
\proof Let $A_n=\xi_n\left(\Z^{2^n}\right)$. In view of
Lemma~\ref{embedding}, Lemma~\ref{subgroup}, and Corollary~\ref{mainL} we obtain $$
\di{\left(\lambda,e^{\alpha\lambda}\right)}\left(F\right)\ge\di{\left(\lambda,e^{\alpha\lambda}\right)}\left(A_n\right)\ge
\di{\left(\frac{\lambda-C_2}{C_1n},C_1n\left(e^{\alpha\lambda}+C_2\right)\right)}\bigoplus_{i=1}^{2^n}\Z=2^n$$
provided
$\frac{\lambda-C_2}{C_1n}\ge 5$  and $C_1n\left(e^{\alpha\lambda}+C_2\right)<2^{\frac{n}{2}-2}$. This holds for
$n=\frac{\lambda-C_2}{5C_1}$ and some $\alpha$.

\endproof

\begin{remark} \label{r:56} It is not known whether the dimension growth of $F$ (with no control) is exponential.
It seems that  the embedding of $\Z^{2^n}$ into $F$ described here has almost the smallest possible quasi-isometry
constants. It does not look like similarly distorted embeddings of $F^{2^n}$ or even $(F\wr Z)^{2^n}$ into $F$ (which
can be defined as above) help proving that the dimension growth of $F$ is exponential. Whether there are less distorted
copies of $\Z^{k}$ or $F^k$ inside $F$ is an open problem. On the other hand, if $F$ has in fact a subexponential
dimension growth, then from Ozawa \cite{Ozawa}, it would follow that $F$  has Guoliang Yu's property A. This would
solve a very difficult open problem (see the Introduction).

Note that if for some $\lambda_0>0$ and $\alpha>0$, $\di{\lambda_0}(\Z^k)\ge k^\alpha$ for every $k$ (see Question
\ref{i}), then the dimension growth of $F$ is exponential. Indeed, in that case the $(C_1n,C_2)$-quasi-isometric
embedding of $\Z^{2^n}$ into $F$ described above gives by Lemmas \ref{shift} and \ref{subgroup} the following
inequalities (for some $C_3>0$):

$$
\di{\left(\lambda_0C_1n+C_2\right)}\left(F\right)\ge
\di{\lambda_0}\left(\Z^{2^n}\right)\ge C_3\left(2^n\right)^\alpha=C_32^{\alpha n}$$ for every $n\ge 1$.
Hence $$\di{n}\left(F\right)\ge 2^{C_4n}$$ for some $C_4>0$ and all $n\ge 1$.
\end{remark}

\section{Upper bounds. The Kolmogorov-Ostrand dimension growth}\label{above}

In order to estimate dimension growth from above, we will use another function which can be traced back to the work of Kolmogorov and Ostrand on Hilbert's
13-th problem~\cite{K,Ost}. Let $X$ be a metric space. Consider colorings of $X$ in where every point can be colored in several colors. Let $\lambda>0$. The definition of monochromatic $\lambda$-clusters remains the same.

\begin{defin}\label{DefKO}
Let $D\colon \N\times\N\to \N$ be a function which is  non-decreasing in each variable.
For every $\lambda>0$, we say that the Kolmogorov-Ostrand (KO) $\lambda$-dimension of $X$ does not exceed $n$  with control $D$ if for every $m\ge 0$ there exists a coloring of $X$ in $m+n$ colors such that every $\lambda$-cluster has diameter at most $D\left(m\right)$ and every point is colored in at least $m+1$ colors. The smallest such $n$ is called the KO $\lambda$-{\em dimension} of $X$ with control $D\left(m,\lambda\right)$ (written as $\diko{\left(\lambda,D\right)}\left(X\right)$).
If the diameters of clusters are uniformly bounded and a control function is not specified, one obtain the notion of dimension $\diko{\lambda}\left(X\right)$.
\end{defin}

\begin{ex}\label{example} Let $D\left(m,\lambda\right)=2\left(m+1\right)\left(\lambda+1\right)$. Then $$\diko{\left(\lambda,D\right)}\left(\R\right)\le 2.$$
\end{ex}
\begin{proof} Indeed, let $m\ge 0$. We color every $x\in \R$ in color $i\in \{0,...,m+1\}$ if $$\left\lfloor \frac{x}{\lambda+1}\right\rfloor\not\equiv 2i+1 \mod 2m+2.$$ Then every point $x$ is colored in $m+1$ colors if $\lfloor \frac{x}{\lambda+1}\rfloor$ is odd and is colored in all $m+2$ colors otherwise. Every $\lambda$-cluster of every color is an interval of size $\left(2m+1\right)\left(\lambda+1\right)< 2\left(m+1\right)\left(\lambda+1\right)$.
\end{proof}

\begin{prop}\label{connection} For every metric space $X$,
$$\di{\lambda}\left(X\right)\le\diko{\lambda}\left(X\right)-1.$$
For every control function $D:\N\times\N\to \N$ which is non-decreasing in each variable  $$\di{\left(\lambda,D\left(0,\lambda\right)\right)}\left(X\right)\le \diko{\left(\lambda,D\right)}\left(X\right)-1.$$
\end{prop}

\proof Indeed, let $\diko{\left(\lambda,D\right)}\left(X\right)\le n\left(\lambda\right)$ for some fixed $\lambda$. Take $m=0$. Then there exists a coloring of $X$ in $n\left(\lambda\right)$ colors such that every $\lambda$-cluster has diameter at most $D\left(0,\lambda\right)$. Then $\di{\left(\lambda,D\left(0,\lambda\right)\right)}\left(X\right)\le n\left(\lambda\right)-1$.\endproof

We always assume that the product of two metric spaces $X\times Y$ is supplied with the $\ell_1$-metric.

\begin{prop}\label{KO-product}
Suppose that $$\diko{\left(\lambda,D_X\right)}\left(X\right)\le n_X\left(\lambda\right)\ \ \text{and}\ \ \diko{\left(\lambda,D_Y\right)}\left(Y\right)\le n_Y\left(\lambda\right).$$ Then $$\diko{\left(\lambda,D\right)}\left(X\times Y\right)\le n\left(\lambda\right)=n_X\left(\lambda\right)+n_Y\left(\lambda\right)-1$$
with $D\left(m,\lambda\right)=D_X\left(m+n_Y\left(\lambda\right),\lambda\right)+D_Y\left(m+n_X\left(\lambda\right),\lambda\right)$.
\end{prop}
\begin{proof} Fix $\lambda>0$ and $m\in \N$. By assumption there exists a coloring $X=\cup_{i=1}^{m+n\left(\lambda\right)} U^{\left(i\right)}$ (resp. $Y=\cup_{i=1}^{m+n\left(\lambda\right)}V^{\left(i\right)}$) where the diameter of every $\lambda$-cluster does not exceed $D_X\left(m+n_Y\left(\lambda\right),\lambda\right)$ (resp. $D_Y\left(m+n_X\left(\lambda\right),\lambda\right)$) and every point is colored in $m+n_Y\left(\lambda\right)+1$ (resp. $m+n_X\left(\lambda\right)+1$) colors. Consider the following coloring of $X\times Y$ in $m+n\left(\lambda\right)$ colors:
$$X\times Y=\cup \left(U^{\left(i\right)}\times V^{\left(i\right)}\right)$$
(i.e. we color each $U^{\left(i\right)}\times V^{\left(i\right)}$ in color $i$). Note that every $\lambda$-cluster of that coloring is the direct product of a cluster $U$ of $U^{\left(i\right)}$ and a cluster $V$ of $V^{\left(i\right)}$ for some $i$.
Since $\diam\left(U\times V\right)=\diam U+ \diam V$,
the diameter of every $\lambda$-cluster is at most $$D_X\left(m+n_Y\left(\lambda\right),\lambda\right)+D_Y\left(m+n_X\left(\lambda\right),\lambda\right).$$

Now pick any point $\left(x,y\right)\in X\times Y$. Note that by our assumption, $x$ is colored by some colors $i\in I$, $|I|\ge m+n_Y\left(\lambda\right)+1$ and $y$ is colored by some colors $j\in J$, $|J|\ge m+n_X\left(\lambda\right)+1$. Since $\left(m+n_X\left(\lambda\right)+1\right)+\left(m+n_Y\left(\lambda\right)+1\right)=m+1+\left(m+n\left(\lambda\right)\right)$, the intersection of $I$ and $J$ has at least $m+1$ numbers. Hence $\left(x,y\right)$ is colored in at least $m+1$ colors as required.
\end{proof}

\begin{prop}\label{KO-power} Suppose that $\diko{\left(\lambda,D\right)}\left(X\right)\le n\left(\lambda\right)$. Then for every $k\in\N$,
$$
\diko{\left(\lambda,D_k\right)}\left(X^k\right)\le kn\left(\lambda\right)-k+1$$
where $D_k\left(m,\lambda\right)=kD\left(m+\left(k-1\right)n\left(\lambda\right),\lambda\right)$.
\end{prop}
\begin{proof}
Induction on $k$. It is true formula for $k=1$.

Assume that $$\diko{\left(\lambda,D_{k-1}\right)}\left(X^{k-1}\right)\le \left(k-1\right)n\left(\lambda\right)$$
for $D_{k-1}\left(m,\lambda\right)=\left(k-1\right)D\left(m+\left(k-2\right)n\left(\lambda\right),\lambda\right)$.
We apply Lemma~\ref{KO-product} to $X^{k-1}\times X$ to obtain
$$\diko{\left(\lambda,D\right)}\left(X^k\right)\le
n\left(\lambda\right)=\left(k-1\right)n\left(\lambda\right)-\left(k-1\right)+1+n\left(\lambda\right)-1
=kn\left(\lambda\right)-k+1$$
with $$D\left(m,\lambda\right)=D_{k-1}\left(m+n\left(\lambda\right)y,\lambda\right)+D\left(m+\left(k-1\right)n\left(\lambda\right)-k+1,\lambda\right)$$
$$=\left(k-1\right)D\left(m+n\left(\lambda\right)+\left(k-2\right)n\left(\lambda\right),\lambda\right)+D\left(m+\left(k-1\right)n\left(\lambda\right)-k+1,\lambda\right)$$
$$=\left(k-1\right)D\left(m+\left(k-1\right)n\left(\lambda\right),\lambda\right)+D\left(m+\left(k-1\right)n\left(\lambda\right)-k+1,\lambda\right)$$
$$\le kD\left(m+\left(k-1\right)n\lambda,\lambda\right)=D_k\left(m,\lambda\right).$$
\end{proof}
\begin{cor}\label{AN-Z}
$\di{\left(\lambda,\left(4n^2-2n\right)\left(\lambda+1\right)\right)}\left(\Z^n\right)\le n$.
\end{cor}
\begin{proof}
We apply this proposition to Example~\ref{example} where $n\left(\lambda\right)=2$ and $$D\left(m,\lambda\right) = 2\left(m+1\right)\left(\lambda+1\right)$$ to obtain
$\diko{\left(\lambda,D_n\right)}\left(\Z^n\right)\le 2n-n+1=n+1$ for $D_n\left(m,\lambda\right)$. By Lemma \ref{connection}
$$\di{\left(\lambda,D_n\left(0,\lambda\right)\right)}\left(\Z^n\right)\le n.$$ Note that
$$D_n\left(0,\lambda\right)=nD\left(2\left(n-1\right),\lambda\right)=2n\left(2n-1\right)\left(\lambda+1\right)=\left(4n^2-2n\right)\left(\lambda+1\right).$$
\end{proof}
This Corollary shows that the constant $c$ in the definition of the Assouad-Nagata dimension of $\Z^n$ has a quadratic in $n$ upper bound.
Will Sawin \cite{Saw} improved this constant to  $\frac{\left(n-1\right)\left(n+2\right)}{2}$, still quadratic.

We recall that the wreath product $B\wr A$ is a semi-direct product of the restricted power $B^{\left(A\right)}$ and $A$.
In case of fixed generating sets for $A$ and $B$ the group $B\wr A$ is equipped with a natural metric $d_{B\wr A}$ described in Theorem~\ref{th:Parry}
Let $\rho$ denote the restriction of this metric to $B^{\left(A\right)}$.

We take the $\ell_1$ metric on the product of metric spaces.
\begin{prop}\label{QI}
The {metric spaces} $\left(B\wr A, d_{B\wr A}\right)$ and $\left(B^{\left(A\right)},\rho\right)\times A$ are quasi-isometric.
\end{prop}
\begin{proof} Given $a\in A$ and a finite set $a_1,\dots a_n\in A$, we consider a shortest path $p$
in $A$ (to be precise - in the Cayley graph of $A$) from $e$ to $a$ that visits all $a_i$s and a shortest loop
 $\omega$ in $A$ based at $e$ that visits all $a_i$s. Let $|p|$ and $|\omega|$ denote the lengths of $p$ and $\omega$ respectively.
We note that $|p|\le|\omega|+\|a\|_A$, $|\omega|\le 2|p$, and $\|a\|_A\le|p|$. Therefore, $|\omega|+\|a\|_A\le 3|p|$ .

Then by Parry's theorem  for $\bar ba\in B^{\left(A\right)}A= B\wr A$, $\bar b=b_1^{a_1}\dots b_n^{a_n}\in B^{\left(A\right)}$, $a,a_i\in A$, $i=1,\dots n$
$$\|\bar ba\|_{B\wr A}=|p|+\sum_{i=1}^n\|b_i||_B\ \ \text{and}\ \ \ \|\bar b\|_{B\wr A}=|\omega|+\sum_{i=1}^n\|b_i||_B.$$
Thus,
$$
\|\bar ba\|_{B\wr A}\le\|\bar b\|_{B\wr A}+\|a\|_A\le 3\|\bar ba\|_{B\wr A}.
$$
\end{proof}

{The proof of the following statement is the same as that for Lemma~\ref{subgroup}.}
\begin{prop}\label{p:subKO}
Let $\phi:X\to Y$ be a quasi-isometric embedding:
$$\frac{1}{c_1}d_X\left(x,x'\right)-r_1\le d_Y\left(\phi\left(x\right),\phi\left(x'\right)\right)\le cd_X\left(x,x'\right)+r $$ for
all $x,x'\in X$ where $c,c_1\ge 1, r,r_1\ge 0$. Then for $\lambda>r$,
$$
\diko{\left(\lambda,D\right)}(Y)\ge\diko{\left(\frac{\lambda-r}{c},c_1\left(D+r_1\right)\right)}(X).
$$
\end{prop}

\begin{prop}\label{Lipschitz iso}
Let $\phi:G'\to G$ be a 1-Lipschitz homomorphism of  finitely generated groups. Then
$\diko{\lambda}\left(G'\right)\le\diko{\lambda}\left(G\right)$.

If additionally for all $x\in G'$,  $\phi$ satisfies the condition
$\|x\|_{G'}\le\|\phi(x)\|_G+r$.
Then $\diko{\left(\lambda,D\right)}\left(G'\right)\le\diko{\left(\lambda, D+r\right)}\left(G\right)$.
\end{prop}
\begin{proof} The first statement is proved the same way as Lemma \ref{p:1-Lip}. The second statement follows from Lemma \ref{p:subKO}.
\end{proof}

\begin{lemma}\label{l:sep} Let $G$ be a (not necessarily finitely generated) group with a proper left invariant metric
$\dist\left(.,.\right)$. Let $U$ be the ball {in} $G$ of radius $\lambda$ {centered at $e\in G$} and let $H$ be a subgroup containing $U$. Then left
cosets of $H$ are $\lambda$-separated.
\end{lemma}
\begin{proof}
Indeed, if $xh, x'h'$ are in two different left cosets ($h,h'\in H$, $x,x'\in G$) but
$$\dist\left(xh, x'h'\right)\le \lambda,$$
then $\dist\left(1,h\iv x\iv x'h'\right)\le \lambda$, so $h\iv x\iv x'h'\in U\subseteq  H$. Hence $x\iv x'\in H$, and $xH=x'H$, a
contradiction.
\end{proof}

We denote by $\rho$ the word metric on $B\wr\Z$ restricted to $B^{\left(\Z\right)}$.
\begin{prop}\label{main prop}
Suppose that $\diko{\left(\lambda,D\right)}\left(B\right)\le n\left(\lambda\right)$. Then $$\diko{\left(\lambda,D'\right)}\left(B^{\left(\Z\right)},\rho\right)\le \left(2\lambda+1\right)\left(n\left(\lambda\right)-1\right)+1$$
with $ D'\left(m,\lambda\right)=\left(2\lambda+1\right)D\left(m+2\lambda n\left(\lambda\right),\lambda\right)+4\lambda+1$.
\end{prop}
\begin{proof} Note that the $\lambda$-ball $U$ in $\left(B^{\left(\Z\right)},\rho\right)$ centered at $e$ generates a subgroup of $B^{2\lambda+1}=B^{[-\lambda,\lambda]\cap \Z}\subset B^{\left(\Z\right)}$.
By Lemma~\ref{l:sep} every $D'$-controlled $\lambda$-coloring  of $\left(B^{[-\lambda,\lambda]\cap \Z},\bar d\right)$ defines by means of translations by elements of $B^{(\Z)}$
a $D'$-controlled $\lambda$-coloring  of $B^{\left(\Z\right)}$ with the same set of colors. Additionally, if every element of $x\in B_{\lambda}$ is painted by at least $m+1$ color, the same holds true for all $y\in B^{\left(\Z\right)}$. Thus,
$$
\diko{\left(\lambda,D'\right)}\left(B^{\left(\Z\right)},\rho\right)=\diko{\left(\lambda, D'\right)}\left(B^{2\lambda+1},\rho\right).
$$
Since the identity map $i:\left(B^{2\lambda+1},\rho\right)\to\left(B^{s\lambda+1},\ell_1\right)$ is a 1-Lipschitz isomorphism satisfying the condition
$\|x\|_{\rho}\le\|x\|_{\ell_1}+4\lambda+1$
by Lemma~\ref{Lipschitz iso} we obtain
$$
\diko{\left(\lambda,\bar D+4\lambda+1\right)}\left(B^{2\lambda+1},\rho\right)\le\diko{\left(\lambda,\bar D\right)}\left(B^{2\lambda+1},\ell_1\right).$$
By Lemma~\ref{KO-power}
$$
\diko{\left(\lambda,\bar D\right)}\left(B^{2\lambda+1},\ell_1\right)\le\left(2\lambda+1\right)\left(n\left(\lambda\right)-1\right)+1
$$
with $\bar D\left(m,\lambda\right)=\left(2\lambda+1\right)D\left(m+2\lambda n\left(\lambda\right),\lambda\right)$.
\end{proof}

\begin{thm}\label{t:bk} For every $k\in\N$ there are constants $a_k$ and $b_k$ such that
\begin{equation}\label{e:ko}
\diko{\left(\lambda,D_k\right)}\left(B_k\right)\le a_k\lambda^k
\end{equation}
with $D_k\left(m,\lambda\right)=b_k\left(m+\lambda^k\right)\lambda^{k+1}$.
\end{thm}
\begin{proof}
Induction on $k$. For $k=0$ this is Example~\ref{example}.

Assume that $
\diko{\left(\lambda,D_k\right)}\left(B_k\right)\le a_k\lambda^k$
with $D_k\left(m,\lambda\right)=b_k\left(m+\lambda^k\right)\lambda^{k+1}$.
By Lemma~\ref{main prop} and  induction assumption
$$\diko{\left(\lambda,D'\right)}\left(B_k^{\left(\Z\right)},\rho\right)\le\left(2\lambda+1\right)\left(a_k\lambda^k-1\right)+1$$
for $ D'\left(m,\lambda\right)=\left(2\lambda+1\right)D_k\left(m+2\lambda a_k\lambda^k,\lambda\right)y+4\lambda+1$.
By Lemma~\ref{KO-product} $$\diko{\left(\lambda,\bar D\right)}\left(\left(B_k^{\left(\Z\right)},\rho\right)\times\Z\right)\le\left(2\lambda+1\right)\left(a_k\lambda^k-1\right)+2 \le 4a_k\lambda^{k+1}$$
for $$\bar D\left(m,\lambda\right)=D'\left(m+2,\lambda\right)+2\left(m+4a_k\lambda^{k+1}+1\right)\left(\lambda+1\right)=$$
$$\left(2\lambda+1\right)D_k\left(m+2+2\lambda a_k\lambda^k,\lambda\right)+4\lambda+1+2\left(m+4a_k\lambda^{k+1}+1\right)\left(\lambda+1\right)$$
$$\le
3b_k\lambda\left(m+2+2\lambda a_k\lambda^k+\lambda^k\right)\lambda^{k+1}+4a_k\lambda\left(m+\lambda^{k+1}\right)\le 4a_kb_k\left(m+\lambda^{k+1}\right)\lambda^{k+2}.$$
By Lemma~\ref{QI} $\left(B_k^{\left(\Z\right)},\rho\right)\times\Z$ is quasi-isometric to $B_k\wr\Z=B_{k+1}$.
By Lemma \ref{p:subKO}, there are constants $a_{k+1}$ and $b_{k+1}$ such that $$\diko{\left(\lambda, D_{k+1}\right)}\left(B_k\wr\Z\right)\le a_{k+1}\lambda^{k+1}$$ with
$D_{k+1}\left(m,\lambda\right)=b_{k+1}\left(m+\lambda^{k+1}\right)\lambda^{k+2}$.
\end{proof}

\begin{cor}
The dimension growth of $B_k$ is at most $\lambda^k$.
\end{cor}
\begin{proof}
Take $m=0$ {in formula (\ref{e:ko}) of Theorem \ref{t:bk}} and apply Lemma~\ref{connection}.
\end{proof}

\begin{question} What is the dimension growth of $B_k$? The answer is not known even for $k=1$, i.e. for the group $\Z\wr \Z$.
\end{question}

\begin{cor} For every solvable subgroup $G$ of the R.Thompson group $F$, there exist polynomials $\Delta_G(\lambda)$ and $C_G(\lambda)$ such that $\di{(\lambda,\Delta_G(\lambda))}(G)$ is at most $C_G(\lambda)$.
\end{cor}

{\proof Indeed, by the result of Collin Bleak \cite{Bleak}, every solvable subgroup of $F$ embeds into the direct power $(B_k)^m$ for some $k, m$.\endproof}

\end{document}